\documentclass[reqno, 11pt]{amsart}
\usepackage{amsmath, latexsym, amsfonts, amssymb, amsthm, amscd,graphicx,color,epsfig}
\usepackage{graphics,epsf,psfrag}

\usepackage{amsmath, latexsym, amsfonts, amssymb, amsthm, amscd,graphicx,color,epsfig}
\usepackage{graphics,epsf,psfrag}

\usepackage{mathtools, latexsym, amsfonts, amssymb, amsthm, amscd,graphicx,color,epsfig,eurosym}
\usepackage[english]{babel}
\usepackage{graphics,epsf,psfrag,graphicx,epsfig}

\mathtoolsset{showonlyrefs,showmanualtags}

\numberwithin{equation}{section}

   \newtheorem{theorem}{Theorem}[section]
   
    \newtheorem{lem}[theorem]{Lemma}
   \newtheorem{prop}[theorem]{Proposition}
   
     \newtheorem{defi}[theorem]{Definition}

\let\a=\alpha    \let\d=\delta  \let\e=\varepsilon
 \let\g=\gamma     \let\k=\kappa  \let\l=\lambda
      \let\o=\omega      
\let\r=\rho  \let\s=\sigma    
  \let\z=\zeta
\let\D=\Delta   \let\G=\Gamma   
\let\O=\Omega

\renewcommand{\hat}{\widehat}

\newcommand{\cA}{\ensuremath{\mathcal A}} 
\newcommand{\cB}{\ensuremath{\mathcal B}} 
 
\newcommand{\cD}{\ensuremath{\mathcal D}} 
\newcommand{\cE}{\ensuremath{\mathcal E}} 
 
\newcommand{\cG}{\ensuremath{\mathcal G}}

\newcommand{\cL}{\ensuremath{\mathcal L}}

\newcommand{\bbE}{{\ensuremath{\mathbb E}} }

\newcommand{\bbN}{{\ensuremath{\mathbb N}} } 
 
\newcommand{\bbP}{{\ensuremath{\mathbb P}} } 
 
\newcommand{\bbR}{{\ensuremath{\mathbb R}} }

\newcommand{\bbZ}{{\ensuremath{\mathbb Z}} }

\newcommand{\tmix}{T_\mathrm{mix}}

\renewcommand{\O}{\Omega}

\newcommand{\ent}{\mathrm{Ent}}
\newcommand{\cov}{\mathrm{Cov}}

\newcommand{\si}{\sigma}

\renewcommand{\l}{\lambda}
\renewcommand{\a}{\alpha}

\newcommand{\grad}{\nabla}

\renewcommand{\star}{*}

\renewcommand{\u}{\pi} 
\def\thsp{\thinspace}

\newcommand{\var}{{\rm Var} }

  \def\tc{\thsp | \thsp}
\newcommand{\ga}{\alpha}

\newcommand{\gga}{\gamma}            

\newcommand{\gO}{\Omega}

\newcommand{\Z}{\mathbb{Z}}

\newcommand{\ind}{\mathbf{1}}

\newcommand{\Tm}{T_{\text{mix}}}

\begin{document}

\title[Dynamics of a directed $(1+d)-$dimensional polymer]{Convergence to equilibrium for a directed $(1+d)-$dimensional polymer }
\author{Pietro Caputo}
\address{Pietro Caputo \\
Universit\`a  Roma Tre. 
}
\email{caputo@mat.uniroma3.it}
\author{Julien Sohier}
\address{Julien Sohier \\Technische Universiteit Eindhoven
}
\email{sohier@tue.nl}

\thanks{ This work was partially supported by the European Research Council through the ÒAdvanced GrantÓ PTRELSS 228032. We thank Fabio Martinelli for several stimulating discussions.}

\begin{abstract}
We consider a flip dynamics for  directed $(1+d)-$dimensional lattice paths with length $L$. The model can be interpreted as a higher dimensional version of the simple exclusion process, the latter corresponding to the case $d=1$. We
prove that the mixing time of the associated Markov chain scales like $L^2\log L$ up to a $d$-dependent multiplicative constant. The  key step in the proof of the upper bound is to show
that the system satisfies a logarithmic Sobolev inequality on the diffusive scale $L^2$ for every fixed $d$, which we achieve by a suitable induction over the dimension together with an estimate for adjacent transpositions. The lower bound is obtained with a version of Wilson's argument \cite{Wil} for the one-dimensional case.
\end{abstract}

\maketitle

\thispagestyle{empty}

 \section{Introduction}
  Consider the set $\O_L$ of all $\bbZ^d$ paths which start and end at the origin after $L$ steps, where $L$ is an even integer. That is, the set of vectors $\eta=(\eta_0,\dots,\eta_L)$, with $\eta_x\in \bbZ^d$, $\eta_0=\eta_L =0$, with $|\eta_x-\eta_{x+1}|=1$. Alternatively, we can look at $\O_L$ as the set of all directed paths in $(1+d)$ dimensions which start at $(o,0)$ and end at $(o,L)$ where $o$ stands for the origin of $\bbZ^d$.  We interpret a configuration $\eta\in\O_L$ as a directed $(1+d)-$dimensional polymer. 

Consider the  Markov chain where independently, at the arrival times of a Poisson clock with intensity $1$, each site $x\in\{1,\dots,L-1\}$ updates the value of $\eta_x$ with a random $\eta'_x$ chosen uniformly among all possible values of the polymer at that site given the values of the polymer $\eta_y$ at all sites $y\neq x$.  To define this process formally, let $\mu$ denote the uniform probability measure on $\O_L$, and write 
\begin{equation*}
     Q_{x}f(\eta) = \mu(f\,|\,\eta_{y }, y\neq x),
   \end{equation*} 
for the conditional expectation of a function $f:\O_L\mapsto\bbR$ at $x$ given the values $\eta_{y}$ at all vertices $y \neq x$. Then, the process introduced above is the continuous-time Markov chain with infinitesimal generator 
   \begin{equation}\label{dyn}
    \cL f = \sum_{x=1}^{L-1} [ Q_{x}f - f],
   \end{equation} 
   for all functions $f:\O_L\mapsto\bbR$. Note that $\cL$ defines a bounded self-adjoint operator on $L^2(\O_L,\mu)$. Indeed, $\cL$ is a symmetric $|\O_L|\times|\O_L|$ matrix.
   For any  $\si\in\O_L$, let $\eta^\si_t$ denote the polymer configuration at time $t$ when the initial condition is $\si\in\O_L$, so that, for any $t\geq 0$,  $\si,\xi\in\O_L$, the matrix element
$p_t(\si,\xi) := e^{t\cL}(\si,\xi)$
represents the probability of the event $\eta^\si_t=\xi$. Since $\cL$ is irreducible and symmetric, one sees that $\mu$ is the unique invariant measure and 
   $$
   \lim_{t\to\infty}p_t(\si,\xi) =\mu(\xi),
   $$
   for any $\si,\xi\in\O_L$.
%
 The mixing time $\tmix$ is defined by 
        \begin{equation}\label{defmix}
       \tmix = \inf\Big\{ t > 0: \max_{\si \in \O_L} \|p_{t}(\si,\cdot)- \mu\|_{\rm TV} \leq 1/4  \Big\},
      \end{equation} 
       where $\|\cdot\|_{\rm TV}$ denotes the total variation distance:
      \begin{equation}
       \|\nu - \nu'\|_{\rm TV} = \tfrac{1}{2} \sum_{\eta \in \gO_{L}} |\nu(\eta) - \nu'(\eta)|,
      \end{equation}  
        for probabilities $\nu,\nu'$ on $\gO_{L}$. We refer to \cite{LPW} for more background on this standard notion of mixing time.
         
The one-dimensional case $d=1$ has been extensively studied in the past. This model is equivalent to the simple exclusion process on the segment $\{1,\dots,L\}$ with $L/2$ particles.
It was shown by D.B.\ Wilson \cite{Wil}
that the mixing time scales like $L^2\log L$ up to a multiplicative constant. 
More recently, a finer analysis of the mixing time was obtained by H.\ Lacoin \cite{Lacoin},  who showed that as $L\to\infty$
 $$\tmix = \big(\tfrac1{4\pi^2}+o(1)\big)L^2\log L.$$ 
Below, we will consider the higher dimensional case $d>1$, where apparently no estimates of this type have been obtained before. As explained later on, one may interpret this as a suitable exclusion process with $d$ different types of particles.
Our main result is as follows.

 \begin{theorem}\label{mainmix}
For any $d \geq 2$, there exist constants $c,C > 0$ such that the inequality
 \begin{equation}\label{esti}
  c\, L^{2} \log L \leq \tmix \leq C L^{2} \log L
 \end{equation}  
 holds for all $L$.
 \end{theorem}     

The lower bound in \eqref{esti} will be obtained in Section \ref{lower} below by a suitable version of a well known argument of D.B.\ Wilson \cite{Wil} for the one-dimensional case. The upper bound requires more work. A direct coupling argument as in the case $d=1$ analyzed by \cite{Wil} does not seem to be available when $d>1$. An important difference with respect to the case $d=1$ is the lack of standard monotonicity tools. 

The main step in the proof of the upper bound is to show that the process satisfies the logarithmic Sobolev inequality with constants scaling like $L^2$. To be more precise, let 
\begin{equation}\label{dirich}\cE(f,g) = -\mu[f\cL g],\end{equation} 
for $f,g:\O_L\mapsto\bbR$,  denote the Dirichlet form of the process, and define the entropy functional 
\begin{equation}\label{entro}
\ent(f) = \mu[f\log f] - \mu[f]\log \mu[f],
\end{equation}
for $f:\O_L\mapsto\bbR_+$. In Section \ref{upper} we show that 
for any $d$, for any $f:\O_L\mapsto\bbR_+$, for any $L\in2\bbN$, one has the inequality 
 \begin{equation}\label{lsi1}
 \ent(f) \leq c\,L^{2} \cE(\sqrt f,\sqrt f)\,
 \end{equation}  
where $c=c(d)$ is a positive constant. 
Once the bound \eqref{lsi1} is available, the upper bound in  Theorem \ref{mainmix} is obtained 
by an application of the standard estimates relating the constant in the log-Sobolev inequality and the mixing time. 

Diffusive scaling of the constants in the log-Sobolev inequality as in \eqref{lsi1} is well known to hold for the simple exclusion process and for  various generalizations of it; see in particular \cite{Yau_gen} and \cite{CMR}. However, the higher dimensional case considered here is not covered by these 
works. One of the main differences is that the model here has $d$ conservation laws rather than just one. We note that if one is after the weaker Poincar\'e, or spectral gap, inequality, then the diffusive estimate could be obtained by adapting the arguments in \cite{shonan}. However, this would not suffice to prove the desired upper bound on the mixing time. 
To prove \eqref{lsi1} instead, we exploit a recursion over the dimension such that at each step the number of particles of a new type is fixed.  At the final stage of the recursion, the numbers of all particle types have been assigned, and the problem is reduced to the proof of diffusive scaling for the log-Sobolev constant in the setting of card shuffling by  adjacent transpositions. The latter is established in Section \ref{upper} below. 
A high-level description of the whole argument is given in Section \ref{over}.
We remark that the same argument actually proves the upper bound $\tmix\leq CL^2\log L $ for the more general problem where the end point $\eta_L$ of the polymer is fixed at an arbitrary value in $ \bbZ^d$, not necessarily the origin, with constant $C$ independent of the value of $\eta_L$.  For simplicity of notation we have chosen to restrict ourselves to the case $\eta_L=o$. On the other hand it should be remarked that our argument provides a constant $C=C(d)$ that is presumably far from optimal, especially for $d$ large.

We end this introduction by mentioning an interesting open question. Consider the above defined polymer model in the presence of a pinning potential, that is when $\mu$ is modified by 
assigning the weight $\l^{N(\eta)}$ to each configuration $\eta\in\O_L$, where $N(\eta)=\sum_{x=1}^{L-1}\ind_{\eta_x=o}$ stands for the number of contacts of the polymer with the origin $o\in\bbZ^d$, and $\l>0$ is a parameter determining the strength of repulsion ($\l<1$) or attraction ($\l>1$) to the origin.  
It is well known that the polymer undergoes a localization/delocalization phase transition, with critical point $\l_c(d)=1$ for $d=1,2$ and $\l_c(d)>1$ for $d>2$; see \cite[Chapter 3]{GBbook}. The mixing time of the polymer in the presence of pinning was studied in depth in \cite{CMT} and \cite{CLMST} in the case $d=1$, where it was shown among other things that there is a slowdown in the relaxation, with subdiffusive behavior, in the delocalized regime $\l<1$. We conjecture that this phenomenon should disappear as soon as $d>1$ and that the mixing time should stay bounded by $O(L^2\log L)$ for any $\l>0$.

          \section{ Proof of the upper bound.}\label{upper}
         
         \subsection{Representation as a particle system}
         Recall that  $\O_L$
 is the set of $\bbZ^d$ paths of length $L$ that start and end at the origin.  For any $\eta\in\O_L$, let $\zeta=\nabla\eta$  
   denote the vector   
   \begin{equation}\label{defzeta}
   \zeta=(\zeta_1,\dots,\zeta_L)\,,\quad \zeta_x=\eta_{x}-
  \eta_{x-1}\,.
  \end{equation}
Through this map the set $\O_L$ will be identified with the set of vectors 
$$
\textstyle{\Big\{\zeta\in \{e_1,\dots,e_{2d}\}^L\,:\;\sum_{i=1}^L\zeta_i = 0\Big\}},  $$
where $e_{j}$, $j=1,\dots,  d$, denotes the canonical basis of $\Z^{d}$, and for notational convenience we define $e_{j+d}:=-e_j$.
        For $j=1,\dots,d$, and $x=1,\dots L$, we say that site $x$ is occupied by a particle of type $j$ if 
$\zeta_x=e_j$, and by an anti-particle of type $j$ if $\zeta_x=-e_j$. At each site there is either a particle or an anti-particle. Because of the constraint $\eta_L=o$, for every type $j=1,\dots,d$ the number of particles equals the number of anti-particles.
 The dynamics defined by \eqref{dyn} is then interpreted naturally as a particle exchange process with creation-annihiliation of particle/anti-particle pairs as follows. 
 Fix a polymer configuration $\eta$ and let $\zeta$ denote the corresponding gradient vector. 
 Fix a site $x$ to be updated. If $\eta_{x-1}\neq \eta_{x+1}$, then one has $\zeta_x= \eta_{x}-
  \eta_{x-1}=e_j$ and $\zeta_{x+1}= \eta_{x+1}-
  \eta_x=e_\ell$, for some $j,\ell\in\{1,\dots,2d\}$ with $e_\ell\neq -e_j$. Thus, in this case 
 there are two possibilities for the new value of $\eta_x$, one corresponding to the current configuration $\eta$, and one corresponding 
 to the configuration $\eta^{x}$ obtained by swapping the increments $\zeta_x,\zeta_{x+1}$. On the other hand, if the polymer is such that $\eta_{x-1}= \eta_{x+1}$, then one must have $(\zeta_x,\zeta_{x+1})=(e_j,-e_j)$, for some $j\in\{1,\dots,2d\}$. Thus, in this case 
 there are $2d$ possibilities for the new value of $\eta_x$. We call $\eta^{x,*,j}$ the polymer configuration that coincides with $\eta$ at all sites $y\neq x$ and such that $(\zeta_x,\zeta_{x+1})=(e_j,-e_j)$.
  Thus, in this process adjacent particles exchange their positions, and when a particle/anti-particle pair occupies adjacent sites it can be deleted to produce a new  particle/anti-particle pair of a different type. 
     With this notation, the generator \eqref{dyn} takes the form
    \begin{equation}\label{desdyn}
      \cL f(\eta) = \sum_{x=1}^{L-1} c_x(\eta)[f(\eta^{x}) - f(\eta)] 
      +  \sum_{x=1}^{L-1} \sum_{j=1}^{2d}  c^{\star, j}_{x}(\eta)[f(\eta^{x, \star, j}) - f(\eta)],
    \end{equation} 
  where $$c_x(\eta):=\tfrac12\,\ind_{\eta_{x-1}\neq\eta_{x+1}},\qquad c^{\star, j}_{x}(\eta):= \tfrac1{2d}\,\ind_{\eta_{x-1}=\eta_{x+1}}.$$
  The Dirichlet form \eqref{dirich} of this process becomes
   \begin{equation}\label{desdyn2}
     \cE(f,g)  = \tfrac12
     \sum_{x=1}^{L-1} \mu[c_x(\nabla_xf)^2] 
      + \tfrac12 \sum_{x=1}^{L-1} \sum_{j=1}^{2d}  \mu[c^{\star, j}_{x}(\nabla_{x}^{*,j}f)^2],
    \end{equation} 
where we use the notation $\nabla_xf(\eta) = f(\eta^{x}) - f(\eta)$, and $\nabla_{x}^{*,j}f(\eta)=f(\eta^{x, \star, j}) - f(\eta)$. 
 Notice that when $d=1$ only the last term in the right hand side of \eqref{desdyn2} survives and we obtain  the symmetric simple exclusion process; see e.g.\ \cite{CMT}. 
          
          \subsection{ Overview of the proof}\label{over}
                     Let us describe the strategy for the upper bound of Theorem  \ref{mainmix}. 
          Recall the definition of the Dirichlet form \eqref{dirich} and define the log-Sobolev constant 
 \begin{equation}\label{lsc}
   \ga(L)= \inf_f \,\frac{\cE(\sqrt{f},\sqrt{f})}{\ent(f)}\,,   \end{equation}
   where $f$ ranges over all functions $f:\O_L\mapsto \bbR_+$.
   
         We refer to \cite{DiSa} for the following  classical inequality relating $\a(L)$ to $\tmix$:
         \begin{equation}\label{mixlo}
   \Tm \leq \frac{4 + \log(\log(1/\pi^{\star}))}{2\ga(L)},
  \end{equation}         
     with $\pi^{\star} := \min_{x \in \gO_{L}} \mu(x) = |\gO_{L}|^{-1}$. 
      Since $|\gO_{L}|\leq (2d)^L$, it suffices to show that $\a(L)^{-1}= O(L^2)$.
   
   Let $N_i$ denote the number of particles of type $i$:
        \begin{equation}\label{def_ni}
         N_{i}(\eta) = \sum_{x=1}^L \ind_{\zeta_x = e_{i}},
        \end{equation} 
        where as above $\zeta_x=\eta_x-\eta_{x-1}$. 
       The law of the vector $(N_{1},\ldots,N_{d})$ under the uniform distribution $\mu$ is given by
        \begin{equation}\label{lawincre}
     \mu(N_{1} = n_{1},\ldots,N_{d}=n_{d}) 
     =\frac1{|\O_L|}{L \choose L/2} {L/2 \choose n_{1},\ldots,n_{d}}^{2} \\
        \end{equation}  
         where $(n_{1},\ldots,n_{d})$ is a vector of non negative integers satisfying $\sum_{j=1}^{d} n_{j} = L/2$,  
          and ${L/2 \choose n_{1},\ldots,n_{d}}$ is the associated multinomial coefficient. Indeed, \eqref{lawincre} follows easily by counting all possible choices of positions of  $n_i$ particles of type $i$ and $n_i$ anti-particles of type $i$, for $i=1,\dots,d$.         
             
            We denote by $\mu^{(i)}$,  $i=1,\dots,d$, the measure $\mu(\cdot|N_{1},\ldots,N_{i})$ obtained by conditioning $\mu$ on a given value of the numbers of particles of type $1,\dots,i$, and we write
         $\mu^{(0)} = \mu$. Notice that $\mu^{(d-1)}=\mu^{(d)}$ because of the global constraint $\sum_{j=1}^dN_j=L/2$. We write
         $$\ent_{i}(f) = \mu^{(i)}[f\log f] - \mu^{(i)}[f]\log \mu^{(i)}[f]\,,$$ for the entropy 
         with respect to the measure $\mu^{(i)}$. Thus, $\ent_{i}(f)$ is a function of the variables $N_{1},\ldots,N_{i}$. 
         
         By adding and subtracting $\mu[f\log \mu(f|X)]$ in the expression \eqref{entro} one obtains the following standard decomposition of entropy via conditioning on a random variable $X$:
          \begin{equation}\label{deco}
          \ent(f) = \ent( \mu(f|X)) + \mu(\ent(f|X))\,,
          \end{equation}
          where $\ent(f|X)$ is the entropy of $f$ with respect to the measure conditioned on the value of $X$ and $ \mu(f|X)$ is the $X$-measurable function defined by conditional expectation. Applied to the measure $\mu^{(i)}$ with $X=N_{i+1}$, and noting that $ \mu^{(i)}(f|N_{i+1})=\mu^{(i+1)}(f)$, \eqref{deco} gives
      that 
       for any $0 \leq i \leq d-1$, 
       \begin{equation}\label{iter}
        \ent_{i}(f) 
        = \ent_{i}(\mu^{(i+1)}(f)) + \mu^{(i)} [\ent_{i+1}(f)].
       \end{equation} 
       Since $\mu^{(d-1)}=\mu^{(d)}$ we note that $\ent_{d-1}(\mu^{(d)}(f))=0$.
A crucial step in our proof will be the following estimate. 
       \begin{theorem}\label{rec_th}
       For any $d\geq 2$, there exist constants $c_1,c_2>0$ such that for all $i=0,\dots,d-2$: 
        \begin{equation}\label{recur}
      \mu(\ent_{i} (\mu^{(i+1)}(f))) \leq c_1 L^{2} \cE(\sqrt{f},\sqrt{f}) + c_2 \mu(\ent_{i+1}(f)),
     \end{equation} 
     for all $f:\O_L\mapsto\bbR_+$.  \end{theorem}
      Once this result is available we proceed as follows. For $i=0$, \eqref{iter} and \eqref{recur} give the estimate
      $$
      \ent(f)\leq c_1 L^{2} \cE(\sqrt{f},\sqrt{f}) + (1+c_2)\mu(\ent_{1}(f)).
      $$
     If we iterate this reasoning, 
we obtain the estimate
     \begin{equation}\label{recur1}
      \ent(f)\leq c'_1 L^{2} \cE(\sqrt{f},\sqrt{f}) + c_2'\mu(\ent_{d}(f)),
       \end{equation} 
       where $c_1'=c_1\sum_{i=0}^{d-1}(1+c_2)^i$ and $c_2'=(1+c_2)^{d}$.
 It remains to estimate $\mu(\ent_{d}(f))$. 
   \begin{theorem}\label{AT_th}
       There exists a constant $C>0$ such that  
        \begin{equation}\label{recur20}
      \mu(\ent_{d} (f)) \leq C L^{2} \cE(\sqrt{f},\sqrt{f}) ,
     \end{equation} 
     for all $f:\O_L\mapsto\bbR_+$.  
       \end{theorem}
The above theorems then allow us to conclude  that the log-Sobolev constant in \eqref{lsc} satisfies $$\tfrac1{\a(L)}\leq (C\,c_2' + c_1') L^2,$$ which ends the proof. The following subsections are devoted to the proof of Theorem \ref{rec_th} and Theorem \ref{AT_th}.

    \subsection{Log-Sobolev inequality for a brith and death chain}\label{secMi}
The starting point in the proof of Theorem \ref{rec_th} is an application of a criterion for log-Sobolev inequalities in birth and death chains due to L.\ Miclo \cite{Mi}. For this purpose we follow \cite{CMR}. 
%
%
%
     In the definition below, we consider a generic  probability measure $\gga$ on 
      the finite set of integers $S:=\{n_{\min},n_{\min}+1,\ldots, n_{\max} \}$, for some $n_{\max} >n_{\min}$. 
     
   \begin{defi}[Condition {\rm Conv}$(c,\bar n)$]
    
     We say that $\gga$ satisfies the convexity hypothesis with parameters $c > 0$ and $\bar n \in S$, which we denote by {\rm Conv}$(c,\bar n)$, if 
     $c^{-1} \bar n \leq n_{\max} - \bar n \leq c \bar n$ and the same inequality holds for $\bar n - n_{\min}$. Furthermore, for any 
      $n \geq \bar n$:
      \begin{equation}\label{conv1}
       \frac{\gga(n+1)}{\gga(n)} \leq c\, e^{-\frac{n-\bar n}{c\bar n}},
      \end{equation} 
     and   for any $n \leq \bar n$, 
        \begin{equation}\label{conv2}
           \frac{\gga(n-1)}{\gga(n)} \leq c\, e^{-\frac{\bar n-n}{c\bar n}}.         
        \end{equation} 
       Finally for any $n \in S$, 
       \begin{equation}\label{conc}
        \frac{1}{c\sqrt{\bar n}}\, e^{-\frac{c(\bar n-n)^{2}}{\bar n}} \leq \gga(n) \leq  \frac{c}{\sqrt{\bar n}}\, e^{-\frac{(\bar n-n)^{2}}{c \bar n}}.
        \end{equation} 
   \end{defi}
The following useful lemma appears in \cite{CMR}. We use the notation $a\wedge b=\min\{a,b\}$.
  \begin{lem}[Proposition A.5 in \cite{CMR}]\label{cmr_prop}
If $\g$ satisfies {\rm Conv}$(c,\bar n$), then for all functions
      $g:S\mapsto \bbR_+$,
      \begin{equation}\label{logsobbide}
       \ent_{\gga}(g) \leq C \,\bar n \sum_{n=n_{\min}+1}^{n_{\max}} [\gga(n) \wedge \gga(n-1)] \Big(\sqrt{g(n)} - \sqrt{g(n-1)}\Big)^{2},
      \end{equation}    
where the constant $C$ depends only on $c$ and not on $\bar n$.  
      \end{lem}
      We shall prove that  the number of particles of a given type has a distribution with the properties described above.  Fix $i\in\{0,\dots,d-1\}$, and fix nonnegative integers
      $n_1,\dots,n_{i}$ such that $\sum_{j=1}^{i}n_j\leq L/2$.  Set $S=\{0,\dots,L_{i+1}/2\}$, where we define $L_{i+1}:=L - 2 \sum_{j=1}^{i}n_j$. Let $\g$ denote the probability on $S$:
      \begin{equation}\label{defnu}
      \g(n): = \mu(N_{i+1}=n\,|\,N_1=n_1,\dots,N_{i}=n_{i})=\mu^{(i)}(N_{i+1}=n).
      \end{equation}
As an immediate corollary of  Proposition \ref{birthanddeathfin}  in the appendix (simply replace $L$ by $L_i$ and $d$ by $d-i$ there), we have that the measure $\g$ defined in \eqref{defnu} satisfies {\rm Conv}$(c,\bar n$) for some absolute constant $c>0$, with $\bar n:=L_{i+1}/2(d-i)$. 
                   Therefore, for any $f:\O_L\mapsto \bbR_+$, it follows from Proposition \ref{birthanddeathfin} and Lemma \ref{cmr_prop} applied to $g(n)=\mu^{(i)}(f\tc N_{i+1}=n)$, that 
  \begin{equation}\label{recur2}
     \ent_{i} (\mu^{(i+1)}(f)) \leq  \frac{C\,L_{i+1}}{2(d-i)}\sum_{n=1}^{L_{i+1}/2}[\g(n)\wedge\g(n-1)]
     \,\r\left({\mu^{(i)}(f\tc n)},{\mu^{(i)}(f\tc n-1)}\right)\,,
     \end{equation} 
where $C>0$ is a constant and we use the notation $\mu^{(i)}(f\tc n):=\mu^{(i)}(f\tc N_{i+1}=n)$
and
 \begin{equation}\label{def_rho} 
 \rho(a,b):=\big(\sqrt{a} - \sqrt{b}\big)^{2}\,,\qquad a,b\geq 0.
            \end{equation}
              
    To proceed towards the proof of Theorem \ref{rec_th} we now estimate the right hand side of \eqref{recur2}. 
     \subsection{Decomposition of  $\r\left({\mu^{(i)}(f\tc n)},{\mu^{(i)}(f\tc n-1)}\right)$.}\label{partbide}
   We need to introduce some more notation. Suppose $u,v\in\{1,\dots,L\}$ and $\eta\in\O_L$ is such that $\z_u=-\z_v=e_j$, for some $j\in\{1,\dots,2d\}$, where $\z=\nabla\eta$ as in \eqref{defzeta}. For any $\ell\in\{1,\dots,2d\}$,  we define the operator $T_{u,v}^{*,\ell}$ by
   \begin{equation}\label{def_tstar} 
 T_{u,v}^{*,\ell}f(\eta) =
 f(\eta^{*,\ell}_{u,v}),
            \end{equation}
where $  \eta^{*,\ell}_{u,v}$ denotes the configuration $\eta'\in\O_L$ that is equal to $\eta$ except that the pair $(\z_u,\z_v)=(e_j,-e_j)$ has been replaced by the pair $(\z'_u,\z'_v)=(e_\ell,-e_\ell)$. Notice that
  this operation is well defined, producing a valid element of $\O_L$, whenever the configuration $\eta$ satisfies $\z_u=-\z_v$.
  Moreover, for any fixed $u,v\in\{1,\dots,L\}$, any $i\in\{0,\dots,d-2\}$, and $\ell\in  A_{i} := \{i+2, \ldots, d \} \cup \{ d + i+2,\ldots,2d \}$, $n\in\{1,\dots,L_{i+1}/2\}$, 
  from the uniformity of $\mu$ we obtain the change of variable formula:
  \begin{equation}\label{chgvar}
\mu^{(i)}\left(f\,\ind_{\zeta_u=-\z_v=e_{i+1}}\ind_{N_{i+1}=n}\right) = \mu^{(i)}\left(T_{u,v}^{*,i+1}f\,\ind_{\zeta_u=-\z_v=e_{\ell}}\ind_{N_{i+1}=n-1}\right),
\end{equation}  
  for any function $f$. Since $\sum_{u,v=1}^{L} \ind_{\z_{u}=-\z_v=e_{i+1}} = N_{i+1} ^2$ one has 
  \begin{align}\label{mfsw}
    & \mu^{(i)}(f\tc       n)  =\frac{1}{\g(n)}\,\mu^{(i)}\left[ f \,\ind_{N_{i+1}=n}\right] = \frac{1}{n^2\g(n)} \sum_{u,v=1}^{L}  \mu^{(i)}\left[ f \,
\ind_{\z_{u}=-\z_v=e_{i+1}} \ind_{N_{i+1}=n}\right]   \nonumber\\
  & = \frac{1}{2(d-i-1)n^2\g(n)} \sum_{u,v=1}^{L} \sum_{\ell\in A_{i}} \mu^{(i)}\left[ T_{u,v}^{*,i+1}f  \,\ind_{\z_{u}=-\z_v=e_{\ell}} \ind_{N_{i+1}=n-1}\right] \nonumber  \\
  & = \frac{\g(n-1)}{2(d-i-1)n^2\g(n)} \sum_{u,v=1}^{L} \sum_{\ell\in A_{i}} \mu^{(i)}\left[ T_{u,v}^{*, i+1}f  \,\ind_{\z_{u}=-\z_v=e_{\ell}} \,|\,N_{i+1}=n-1\right] .  \end{align}
 We introduce the notation
   \begin{equation}\label{def_chi}
    \begin{aligned}
         & \chi_{u,v,\ell} = \frac{\g(n-1)}{2(d-i-1)n^2\g(n)} \ind_{\z_{u}=-\z_v=e_{\ell}},\\
   & \chi =  \sum_{u,v=1}^{L} \sum_{\ell\in A_{i}}  \chi_{u,v,\ell},
    \end{aligned}
   \end{equation} 
    so that \eqref{mfsw} takes the more compact form 
     \begin{equation}\label{mfsw1}
      \mu^{(i)}(f\tc n) = \sum_{u,v,\ell}  \mu^{(i)}\left( \chi_{u,v,\ell} T_{u,v}^{*,i+1}f  \tc n-1\right).
     \end{equation} 
   Considering the constant function $f \equiv 1$ one has the normalization
        \begin{equation}\label{renor}
     \mu^{(i)}(\chi\tc n-1) =   \sum_{u,v,\ell} \mu^{(i)}\left( \chi_{u,v,\ell} \tc n-1\right) = 1.
      \end{equation} 
   Moreover,  using symmetry, for any $\ell\in A_i$ we have \begin{equation}\label{conv}
  \mu^{(i)}\left(N_{\ell}^{2}\tc n-1\right) = \frac{n^{2}\g(n)}{\g(n-1)}.
     \end{equation} 
     From \eqref{mfsw1}, using $\r(a,b)\leq 2\r(a,c)+2\r(b,c)$, for any $a,b,c\geq 0$, one has
       \begin{align}\label{rdeco}
     \r\left({\mu^{(i)}(f\tc n)},{\mu^{(i)}(f\tc n-1)}\right)
      & \leq  \textstyle{2\r\left(
        \sum_{u,v,\ell}  \mu^{(i)}\left( \chi_{u,v,\ell} T_{u,v}^{*,i+1}f  \tc n-1\right)
        , \mu^{(i)}(\chi f\tc n-1)\right)} \nonumber \\ &\quad\quad \quad+ 2\r\left({\mu^{(i)}(\chi f\tc n-1)},{\mu^{(i)}(f\tc n-1)}\right) \nonumber\\
        & =: \cA_i(f,n) + \cB_i(f,n)\,.
       \end{align}
The contributions of the two terms above to the expression \eqref{recur2} will be analyzed separately.

   \subsection{Estimating $\sum_n [\g(n-1)\wedge\g(n)]\cA_i(f,n)$. }
   Here we prove the following estimate.
   \begin{prop}\label{cin}
   There exists a constant $C>0$ such that for all $L\in2\bbN$, $i=0,\dots,d-1$, 
for all even $L_{i+1}\in[2,L]$, all integers $n\in[1,L_{i+1}/2]$, and for all functions $f:\O_L\mapsto\bbR_+$, one has
   \begin{align}\label{topropo}
  &L_{i+1} \sum_n [\g(n-1)\wedge\g(n)]\cA_i(f,n) \nonumber 
  \\ & \qquad \leq C L^2 \sum_{x=1}^{L-1}
   \mu^{(i)}[c_x(\nabla_x\sqrt f)^2] 
    + CL \sum_{x=1}^{L-1}\sum_{j=1}^{2d}  \mu^{(i)}[c^{\star, j}_{x}(\nabla_{x}^{*,j}\sqrt f)^2].
   \end{align}
   \end{prop}
\begin{proof}   Since $\r:[0,\infty)^2\mapsto\bbR$ is convex, by Jensen's inequality and the expressions \eqref{def_chi} and \eqref{renor},    \begin{align}\label{rconv}
   \cA_i(f,n) \leq 2 \sum_{u,v,\ell}  \mu^{(i)}\left( \r\left(T_{u,v}^{*,i+1}f, f\right) \chi_{u,v,\ell}  \tc n-1\right).
   \end{align}
   We now turn to the estimate of the right hand side of \eqref{rconv}. 
We need to compare the exchanges at positions $u,v$ with local exchanges between adjacent positions.   Fix $u,v$ and assume without loss of generality that $v\geq u+1$.  
 Set $h=\sqrt f$ so that 
$$ \r\left(T_{u,v}^{*,i+1}f, f\right) = \left( \nabla_{u,v}^{*,i+1} h \right)^{2}.$$
The operation 
       $ T_{u,v}^{*,i+1}$ can be implemented by first transferring $\z_{u}$ from position $u$ to position $v-1$,  through a chain of adjacent swaps, then
      applying the operation $T_{v-1,v}^{*,i+1}$ and then finally transferring back  
          the new value of $\z_{v-1}$ from position $v-1$ to position $u$ via adjacent swaps. This can be formalized as follows. 
Let $T_u$ denote the adjacent swap operator that changes $(\z_u,\z_{u+1})$ into $(\z_{u+1},\z_u)$, that is $T_uh(\eta) = h(\eta^u)$ 
          for any function $h:\O_L\mapsto\bbR$; see \eqref{desdyn}. Thus for any function $h$ we can write            
       \begin{align*}
        T_{u,v}^{*,i+1}h & = 
        T_{u} \circ \ldots \circ T_{v-2} \circ T_{v-1,v}^{*,i+1} \circ T_{v-2}
 \circ \ldots \circ T_{u}
  h . 
       \end{align*}
   In terms of the gradient operators $\nabla_u=T_u-1$, $\nabla_{u}^{*,i+1}=T_{u}^{*,i+1}-1$, one has the telescopic sum
         \begin{equation}\label{grdcm}
          \begin{aligned}
    T_{u,v}^{*,i+1}h -h  & = \sum_{j=0}^{v-u-3} 
    \nabla_{u+j} T_{u+j+1}\circ \ldots \circ T_{v-2} \circ T_{v-1,v}^{*,i+1} \circ T_{v-2}
 \circ \ldots \circ T_{u} h\\ 
  &  \quad + \nabla_{v-2} T_{v-1,v}^{*,i+1} \circ T_{v-2}
 \circ \ldots \circ T_{u} h  + \nabla_{v-1}^{*,i+1}T_{v-2}
 \circ \ldots \circ T_{u} h \\
   &  \qquad+ \sum_{j=1}^{v-u-2}  \nabla_{u+j} 
 T_{u+j-1} \circ \ldots \circ T_{u}  h 
   + \nabla_{u} h\,.
   \end{aligned}
         \end{equation} 
       By the uniformity of $\mu^{(i)}$, every term in the first sum in \eqref{grdcm} satisfies
       \begin{align}\label{chgvar0}
    & \mu^{(i)}\left[ \left(\nabla_{u+j} 
 T_{u+j-1} \circ \ldots \circ T_{v-2}\circ T_{v-1,v}^{*,i+1} \circ T_{v-2}\circ\ldots \circ T_{u}  h\right)^{2}  \chi_{u,v,\ell} \tc n-1\right] 
  \nonumber \\
  &\qquad\qquad = \frac{\gga(n)}{\gga(n-1)}\mu^{(i)}\left[ \left(\nabla_{u+j} 
  h\right)^{2}   \chi_{u+j-1,v,i+1} \tc n\right].
   \end{align}  
      The same identity holds for the first term in the second line of \eqref{grdcm}, i.e.\ when $u+j=v-2$.
    In a similar way, for the terms in the last line of \eqref{grdcm}, one obtains \begin{align}\label{chgvar2}
         \mu^{(i)}\left[ \left( \nabla_{u+j}T_{u+j-1}
 \circ \ldots \circ T_{u} h \right)^{2}  \chi_{u,v,\ell} \tc n-1\right] 
  = \mu^{(i)}\left[ \left(\nabla_{u+j} 
  h\right)^{2}   \chi_{u+j-1,v,\ell} \tc n-1\right].
\end{align}
       Finally, for the term involving the gradient $\nabla_{v-1}^{*,i+1}$ one has
   \begin{align}\label{chgvar3}
    \!\!\!\!\!\!  \!\!\!\!\!   \mu^{(i)}\left[ \left( \nabla_{v-1}^{*,i+1}T_{v-2}
 \circ \ldots \circ T_{u} h \right)^{2}  \chi_{u,v,\ell} \tc n-1\right] 
  = \mu^{(i)}\left[ \left(\nabla_{v-1} ^{*,i+1}
  h\right)^{2}   \chi_{v-1,v,\ell} \tc n-1\right].
\end{align}
From \eqref{rconv}-\eqref{chgvar3}, using Schwarz' inequality we have
\begin{align}\label{chgvar4}\cA_i(f,n)& \leq 6L^2 \sum_{x=1}^{L-1}\sum_{v=1}^{L}\sum_{\ell\in A_i} 
\left\{\frac{\g(n)}{\g(n-1)} \mu^{(i)}\left[ \left(\nabla_{x} 
  h\right)^{2}  \chi_{x-1,v,i+1} \tc n\right] + \mu^{(i)}\left[ \left(\nabla_{x} 
  h\right)^{2} \chi_{x-1,v,\ell}\tc n-1\right] \right\}\nonumber\\
   & \;\;\;\qquad+6L\sum_{x=1}^{L-1}\sum_{\ell\in A_i}
\mu^{(i)}\left[ \left(\nabla_{x} ^{*,i+1}
  h\right)^{2}   \chi_{x,x+1,\ell} \tc n-1\right].
\end{align}
Recalling \eqref{def_chi}, one has that for any $x$:
\begin{gather*}
\frac{\gga(n)}{\gga(n-1)} \sum_{v=1}^{L}\sum_{\ell\in A_i} \chi_{x-1,v,i+1} \leq 
\frac{N_{i+1}}{n^2}\,,
\\
\sum_{v=1}^{L}\sum_{\ell\in A_i} \chi_{x-1,v,\ell} \leq \frac{\g(n-1)}{n^{2}\g(n)}  L_{i+1}\,,
 \\
\sum_{\ell\in A_i}\chi_{x,x+1,\ell}\leq 
\frac{\g(n-1)}{n^2\g(n)}\ind_{\z_{x}=-\z_{x+1}}.
\end{gather*}
Next, we claim that 
\begin{gather}\label{claims}
 \frac{L_{i+1}}{n}[\g(n-1)\wedge\g(n)]\,\leq C\g(n)\\
 \label{claimss}\frac{L_{i+1}^2}{n^2}[\g(n-1)\wedge\g(n)]\,\frac{\g(n-1)}{\g(n)}\leq C\g(n-1),
  \end{gather}
for some constant $C>0$. Since $n\leq L_{i+1}$ it is clear that \eqref{claimss} implies \eqref{claims}. On the other hand \eqref{claimss} follows from  
 $$
 \min\Big\{\frac1{n^2},\frac{\g(n-1)}{n^2\g(n)}\Big\}\leq \frac{C}{L_{i+1}^2}, $$
 which is an immediate consequence of the estimate $\g(n-1)/(n^2\g(n))\leq C(L_{i+1}-2n)^{-2}$; see Lemma \ref{le1} in 
 the appendix (applied with $L$ replaced by $L_{i+1}$ and $d$ replaced by $d-i$).
For the first term in \eqref{chgvar4}, using \eqref{claims} we then obtain
 \begin{align*}
 & L_{i+1}\sum_{n=1}^{L_{i+1}/2} \left(\g(n) \wedge\g(n-1)\right) \sum_{x=1}^{L-1}\sum_{v=1}^{L}\sum_{\ell\in A_i} \frac{\g(n)}{\g(n-1)} \mu^{(i)}\left[ \left(\nabla_{x} 
  h\right)^{2}  \chi_{x-1,v,i+1} \tc n\right] \\ & \;\;\;\qquad \leq \sum_{n=1}^{L_{i+1}/2}  \frac{L_{i+1}}{n}[\g(n-1)\wedge\g(n)] \sum_{x=1}^{L-1}  \mu^{(i)}\left[ \left(\nabla_{x} 
  h\right)^{2}  \tc n\right] \nonumber \\ &  \;\;\;\qquad \leq C \sum_{n=1}^{L_{i+1}/2} \sum_{x=1}^{L-1} \g(n) \mu^{(i)}\left[ \left(\nabla_{x} 
  h\right)^{2}  \tc n\right] \leq C\sum_{x=1}^{L-1}\mu^{(i)}\left[ \left(\nabla_{x} 
  h\right)^{2}  \right].
 \end{align*}
For the second  term in \eqref{chgvar4}, using \eqref{claimss} we have
\begin{align*}
 & L_{i+1}\sum_{n=1}^{L_{i+1}/2} \left(\g(n) \wedge\g(n-1)\right) \sum_{x=1}^{L-1}\sum_{v=1}^{L}\sum_{\ell\in A_i} \mu^{(i)}\left[ \left(\nabla_{x} 
  h\right)^{2} \chi_{x-1,v,\ell}\tc n-1\right] \\ & \;\;\;\qquad \leq \sum_{n=1}^{L_{i+1}/2}  \frac{L_{i+1}^2}{n^2}[\g(n-1)\wedge\g(n)]\,\frac{\g(n-1)}{\g(n)} \mu^{(i)}\left[ \left(\nabla_{x} 
  h\right)^{2}  \tc n-1\right] \nonumber \\ &  \;\;\;\qquad \leq C \sum_{n=1}^{L_{i+1}/2} \sum_{x=1}^{L-1} \g(n-1) \mu^{(i)}\left[ \left(\nabla_{x} 
  h\right)^{2}  \tc n-1\right] \leq C\sum_{x=1}^{L-1}\mu^{(i)}\left[ \left(\nabla_{x} 
  h\right)^{2}  \right].
 \end{align*}
Similarly, the last term in  \eqref{chgvar4} satisfies
\begin{align*}
 & L_{i+1}\sum_{n=1}^{L_{i+1}/2} \left(\g(n) \wedge\g(n-1)\right) 
 \sum_{x=1}^{L-1}\sum_{\ell\in A_i}
\mu^{(i)}\left[ \left(\nabla_{x} ^{*,i+1}
  h\right)^{2}   \chi_{x,x+1,\ell} \tc n-1\right]
 \\& \qquad \qquad \qquad\qquad \qquad \qquad\qquad    \leq C
 \sum_{x=1}^{L-1}
\mu^{(i)}\left[ \left(\nabla_{x} ^{*,i+1}
  h\right)^{2} \right].
 \end{align*}
 This ends the proof of Proposition \ref{cin}.
    \end{proof}

   \subsection{
   Covariance estimate }
    Here we estimate the contribution of the second term in \eqref{rdeco}.
   \begin{prop}\label{cino}
   There exists a constant $C>0$ such that for all $L\in2\bbN$, $i=0,\dots,d-2$, 
for all even $L_{i+1}\in[2,L]$, all integers $n\in[1,L_{i+1}/2]$, and for all functions $f:\O_L\mapsto\bbR_+$, one has
   \begin{align}\label{toprop}
  &L_{i+1} \sum_n [\g(n-1)\wedge\g(n)]\cB_i(f,n)
  \leq 
    C \,\mu^{(i)}\left(\ent_{i+1}(f)\right).
   \end{align}
   \end{prop}
\begin{proof}  Note that we may assume that $i\leq d-3$ here since otherwise the function $\chi$ in \eqref{def_chi} is deterministically equal to $1$ under the measure $\mu^{(i)}[\cdot\tc n-1]$, and therefore $\cB_i(f,n)=0$ for all $n$ and $f$. 
Using the
     inequality $$ \left(\sqrt{x} - \sqrt{y}\right)^{2} \leq \frac{\left(x-y\right)^{2}}{x \vee y},$$ valid for $x,y \geq 0$, one has 
     \begin{equation}\label{bi1}
     \cB_i(f,n)\leq 2\,\frac{\mu^{(i)}(f(\chi-1)\tc n-1)^2}{\mu^{(i)}(f\tc n-1)}\,.
     \end{equation} 
     Since by \eqref{renor} $\chi$ satisfies $\mu^{(i)}(\chi\tc n-1)=1$, one has
     $$\mu^{(i)}(f(\chi-1)\tc n-1)=\cov_i(f,\chi\tc n-1),$$
     where $\cov_i(\cdot\tc n-1)$ denotes covariance with respect to $\mu^{(i)}(\cdot\tc n-1)$.
     Let us define $$
\D(n,L_{i+1})=\frac{1}L_{i+1}\,\Big(1\vee\frac{\g(n-1)}{\g(n)}\Big).
$$
We are going to prove that for some constant $C>0$ one has
\begin{equation}\label{top}
\cov_i(f,\chi\tc n-1)^2\leq C\,\D(n,L_{i+1})\,\mu^{(i)}(f\tc n-1)
\ent_i(f\tc n-1),
\end{equation}
where $\ent_i(\cdot\tc n-1)$ stands for the entropy with respect to $\mu^{(i)}(\cdot\tc n-1)$.

If \eqref{top} holds, then \eqref{bi1} implies 
$$
L_{i+1}  [\g(n-1)\wedge\g(n)] \cB_i(f,n)\leq 2C
\g(n-1)\ent_i(f\tc n-1).
$$
Using $$\sum_{n}\g(n-1)\ent_i(f\tc n-1) = \mu^{(i)}(\ent_{i+1}(f)),$$ 
we obtain the desired inequality \eqref{toprop}. Thus, it suffices to prove \eqref{top}.

To prove \eqref{top}, by homogeneity, we may assume without loss of generality that
$ \mu^{(i)}(f\tc n-1)=1$. In Proposition \ref{laptr} below we establish that for some constant $C_1>0$ one has the Laplace transform bound
\begin{equation}\label{laptr1}
\log\mu^{(i)}\left(e^{t(\chi-1)}\tc  \,n-1\right)\leq C_1\,t^2\,\D(n,L_{i+1})\,,\qquad t\in\bbR.
\end{equation}
  We remark that \eqref{top} follows easily from \eqref{laptr1}. 
Indeed, set for simplicity $\nu:=\mu^{(i)}(\cdot\tc n-1)$ and write $\ent_\nu(\cdot)$ for the corresponding entropy. The  variational principle for entropy implies 
that for any $f\geq 0$ with $\nu(f)=1$ one has
$$\ent_\nu(f)=\nu(f\log f) \geq \nu(fh) - \log \nu(e^h)\,,
$$
for any function $h$. Therefore
$$
\nu(f(\chi-1))\leq \tfrac1s\,\ent_\nu(f)+ \tfrac1s\log \nu\left(e^{s(\chi-1)}\right)\leq \tfrac1s\,\ent_\nu(f)
+ C_1\,s\,\D\,,
$$
for all $s>0$, where we write $\D=\D(n,L_{i+1})$ and we use \eqref{laptr1} with $t=s$. Using also \eqref{laptr1} with $t=-s$ one concludes that
$$
|\nu(f(\chi-1))|\leq  \tfrac1s\,\ent_\nu(f)
+ C_1\,s\,\D\,,
$$
for all $s>0$. 
Setting $s=\sqrt{\D^{-1}\ent_{\nu}(f)}$ one 
obtains 
$$
\nu(f(\chi-1))^2\leq (1+C_1)^2\D\,\ent_{\nu}(f),
$$
which is the desired estimate \eqref{top}. It remains to prove \eqref{laptr1}.
\end{proof}

\subsection{Laplace transform estimate}\label{laptrsec}
Here we prove \eqref{laptr1}. We use as above the shorthand notation $\nu=\mu^{(i)}(\cdot\tc n-1)$ and $\D=\D(n,L_{i+1})$.
   \begin{prop}\label{laptr}
   There exists a constant $C_1>0$ such that for all $t\in\bbR$ one has
      \begin{align}\label{laptr2}
      \log\nu\left(e^{t(\chi-1)}\right)\leq C_1\,t^2\,\D.
   \end{align}
   \end{prop}
\begin{proof}
From \eqref{def_chi}-\eqref{conv} we have
$$
\chi -1= \frac{\sum_{j=i+2}^{d} (N_{j}^{2}-\nu[N_{j}^{2}])}{ (d-i-1) \nu[N_{i+2}^{2}]}.
$$
Define the centered variables $\bar N_j:=N_j-\nu[N_{j}]$, where $\nu[N_{j}]=(L_{i+1}-2(n-1))/2(d-i-1)$. Observe that 
by the conservation laws one has $\sum_{j=i+2}^{d} N_j = (L_{i+1}-2(n-1))/2$ and therefore 
$\sum_{j=i+2}^{d} \bar N_j = 0$.  
From these relations we see that 
$$
\chi -1=\frac1{d-i-1}\sum_{j=i+2}^{d}Y_j,\qquad Y_j:= \frac{\bar N_j^2 - \si^2}{ \nu[N_{j}^{2}]},
$$
where we define the variance  $\si^2:= \nu[\bar N_j^2]$. 
Using 
 H\"older's inequality, we get:
 \begin{align*}
 \nu\left(e^{t\left(\chi  -1\right)} \right) \leq \nu \left(e^{tY_{i+2} } \right)^{\tfrac{1}{d-i-1}}
            \nu\left(e^{\tfrac{t}{d-i-2}\sum_{j=i+3}^{d}Y_j} \right)^{\tfrac{d-i-2}{d-i-1}}.
   \end{align*}
Since all $Y_j$ have the same distribution under $\nu$, say $Y:=Y_{i+2}$, iterating the above inequality we see that it is sufficient to prove that there exists  some constant $C>0$ such that for all $t\in\bbR$ one has 
 \begin{align}\label{unidim}
      \log\nu\left(e^{tY}\right)\leq C\,t^2\,\D.
   \end{align}
The proof of \eqref{unidim} is divided into several steps, corresponding to different sets of values for the parameters $t$ and $n$.  

For simplicity, we only consider the case $t\geq 0$. The case $t\leq 0$ follows with the very same arguments. 
We often write $C,C_1,C_2,\dots$ for positive constants that are independent of the parameters $n,L_{i+1},L$ etc.\ but may depend on $d$.   Their value may change from line to line.

From Lemma \ref{le2} in the appendix we know that $\si^2$ is proportional to $(L_{i+1} - 2n)$. Notice that   \begin{align}\label{secm}
\nu[N_{i+2}^2]\geq \nu[N_{i+2} ]^2 =\left(\tfrac{L_{i+1}-2(n-1)}{2(d-i-1)}\right)^{2},
\end{align}
and that $N_{i+2}^2\leq (L_{i+1}-2(n-1))^2$. Therefore 
\begin{align}\label{immm}
Y\leq 4(d-i-1)^{2}.
\end{align}
Suppose that $t\geq a\,\D^{-1}$ for a fixed constant $a>0$. Then $$\nu\left(e^{tY}\right)\leq e^{4d^{2}t}\leq e^{4d^{2}t^2\D /a},$$
which implies \eqref{unidim} with $C = 4d^{2}/a$.

Next, assume that $t\leq b$ for some fixed constant $b>0$. From \eqref{secm} and Lemma \ref{le2} applied to the variable $X=(\bar N_{i+2})^2-\si^2$ in the system of size $L_{i+1}-2(n-1)$, we have that $$\var_\nu(Y)\leq C(L_{i+1}-2n)^{-2}.$$
From Lemma \ref{le1} applied to the system of size $L_{i+1}$, one has that 
$$
\D\geq c\,(L_{i+1}-2n)^{-1},
$$
for some positive constant $c$.
Combining these facts with the well known inequality $$\nu(e^h)\leq \exp{\big(\tfrac12\,\nu[h^2e^{|h|}]\big)},$$ 
which is valid for any function $h$ with $\nu(h)=0$  (use  $e^a\leq   1 + a + \tfrac12a^2e^{|a|}$ and $1+x\leq e^x$), we get
$$
\nu\left(e^{tY}\right)\leq \exp{(t^2C_1(L_{i+1}-2n)^{-1}e^{4d^2b})}\leq e^{C_2\D t^2}.
$$
Thus, we have shown that \eqref{unidim} holds for all $t\leq b$ and $t\geq a\D^{-1}$, and we have freedom on the choices of $a$ and $b$. 
In particular, we can  consider $a$ small and $b$ large if we wish.

Next, we observe that \eqref{unidim} holds for all $t\geq 0$ if $(L_{i+1}-2n)\leq \sqrt {L_{i+1}}$.
Indeed, from Lemma \ref{le1}, we know that $\D\geq c>0$ for some $c>0$ in this case. Therefore,
taking suitable constants $a,b$ (that is $a$ small and $b$ large enough) we cover all $t\geq 0$ with the previous argument. 

Thus, we are left with the case $ (L_{i+1}-2n)\geq \sqrt {L_{i+1}}$ for all $t\in[b,a\D^{-1}]$. Since by Lemma \ref{le1} one has $\D^{-1}\leq C(L_{i+1}-2n)$ 
for some constant $C>0$, we may actually restrict to $t\in[b,c(L_{i+1}-2n)]$ where $c$ can be made small if we wish.

\begin{lem}\label{le10}
There exists a constant $c>0$ such that for all $n$ satisfying $L_{i+1} -2 n \geq \sqrt {L_{i+1}}$, for all $t\leq  c(L_{i+1}-2n)$
we have 
\begin{equation}\label{abou1}
c\,\nu\left(e^{tY}\right)\leq 1.
\end{equation}
\end{lem}
\begin{proof}
We compute
\begin{equation}\label{co1}
\nu\left(e^{tY}\right) =\sum_k \nu(\bar N_{i+2}=k) \exp{\left(t\,\tfrac{k^2-\s^2}{\nu[N_{i+2}]^2}\right)}.
\end{equation}
Using $t\leq c(L_{i+1}-2n)$ and $\nu[N_{i+2}^2]\geq (L_{i+1}-2n)^{2}/4d^2$, see \eqref{secm}, we can bound 
$$ \exp{\left(t\,\tfrac{k^2-\s^2}{\nu[N_{i+2}]^2}\right)}\leq \exp{\left(\tfrac{4d^2c\,k^2}{L_{i+1}-2n}\right)}.
$$
From Proposition \ref{birthanddeathfin}   we know that
\begin{equation}\label{gauss}
\nu(\bar N_{i+2}=k) \leq \frac{C}{\sqrt{L_{i+1}-2n}}\exp{\left(-\tfrac{k^2}{C(L_{i+1}-2n)}\right)}
\end{equation}
for some constant $C>0$.  Thus, taking $c$ small enough, one has that \eqref{co1} is bounded by a constant. Adjusting the value of constants yields the desired conclusion \eqref{abou1}.
\end{proof}

An immediate consequence of Lemma \ref{le10} is that \eqref{unidim} holds for all $t\in [(L_{i+1}-2n)^{1/2},c(L_{i+1}-2n)]$. Indeed, it
suffices to observe that here 
$$\log\nu\left(e^{tY}\right) \leq C_1\leq C_1t^2/(L_{i+1}-2n)\leq Ct^2\D,$$
for some new constant $C>0$. 

Therefore, for the proof of \eqref{unidim} we are left with the regime $t\in [b,(L_{i+1}-2n)^{1/2}]$, and $ (L_{i+1}-2n)\geq \sqrt {L_{i+1}}$.
Here we use the following two facts.

 \begin{lem}\label{le11}
For any $\d>0$, there exists a constant $c_1>0$ such that for all $n$ satisfying
$L_{i+1} -2 n \geq \sqrt {L_{i+1}}$ and for all $t\leq c_1(L_{i+1}-2n)$ one has 
\begin{equation}\label{abbou1}
\sum_{|k|\geq (L_{i+1}-2n)^{1/2+\d}}\nu(\bar N_{i+2}=k) \exp{\left(t\,\tfrac{k^2-\s^2}{\nu[N_{i+2}]^2}\right)}
\leq \exp{\left(-c_1 (L_{i+1}-2n)^{\d}\right)}
\end{equation}
\end{lem}
\begin{proof}
This follows immediately from \eqref{secm} and \eqref{gauss}.
\end{proof}

\begin{lem}\label{le12}
For any $\d\in(0,\tfrac16)$, there exists a constant $C_1>0$ such that for all $n$ satisfying $L_{i+1} -2 n \geq \sqrt {L_{i+1}}$,
for all $0\leq t\leq (L_{i+1}-2n)^{1/2}$ one has 
\begin{equation}\label{abbou2}
\sum_{|k|\leq (L_{i+1}-2n)^{1/2+\d}}\nu(\bar N_{i+2}=k) \exp{\left(t\,\tfrac{k^2-\s^2}{\nu[N_{i+2}]^2}\right)}
\leq  \exp{\left(\tfrac{C_1t^2}{L_{i+1}-2n}\right)}.
\end{equation}
\end{lem}
\begin{proof}
Set $y_k =\tfrac{k^2-\s^2}{\nu[N_{i+2}]^2}$. We observe that if $|k|\leq (L_{i+1}-2n)^{1/2+\d}$, $\d\in(0,\tfrac16)$, and $t\leq (L_{i+1}-2n)^{1/2}$ 
then $|ty_k|\leq 1$. Then we can expand $$e^{ty_k} \leq  1 + ty_k + \tfrac12 t^2y_k^2 + C|t^3y_k^3|,$$
for some absolute constant $C>0$. Note that $\sum_{k }\nu(\bar N_{i+2}=k)y_k=\nu(Y)=0$. On the other hand, using Lemma \ref{le2}, we get
$$\sum_{k}\nu(\bar N_{i+2}=k)y^2_{k} =\var_\nu(Y)\leq C (L_{i+1}-2n)^{-2}.$$ 
Moreover, for $|k| \leq (L_{i+1}-2n)^{1/2+\d} $, we have the bound $|y_k|^3\leq (L_{i+1}-2n)^{6\d-3}$ and thus, for $t\leq (L_{i+1}-2n)^{1/2}$, one has
 $$|t^3y_k^3|\leq t^2 (L_{i+1}-2n)^{6\d-5/2}\leq  t^2 (L_{i+1}-2n)^{-3/2}.$$
This and Lemma \ref{le11} prove that the left hand side of \eqref{abbou2} is bounded above by $1 + Ct^2(L_{i+1}-2n)^{-1}$ for some new $C>0$. Using the bound $1+x\leq e^x$, this concludes the proof. 
\end{proof}

We can now finish the proof of \eqref{unidim}. Recall that it remained to check the estimate in the case $b\leq t\leq (L_{i+1}-2n)^{1/2}$, and $ (L_{i+1}-2n)\geq \sqrt {L_{i+1}}$. From the two lemmas above we have
\begin{align*}
\nu\left(e^{tY}\right)&\leq \exp{\left(\tfrac{C_1t^2}{L_{i+1}-2n}\right)}+\exp{\left({-c_1 (L_{i+1}-2n)^{\d}}\right)}
\\ & = \exp{\left(\tfrac{C_1t^2}{L_{i+1}-2n}\right)}\Big(1 + \exp{\left({-c_1 (L_{i+1}-2n)^{\d}} -\tfrac{C_1t^2}{L_{i+1}-2n} \right)}\Big).
\end{align*}
Taking logarithms, using the bounds $\log(1+x) \leq x$ and  $t\geq b$ one has
$$
\log\nu\left(e^{tY}\right)\leq\tfrac{C_1t^2}{L_{i+1}-2n} + \exp{\left({-c_1 (L_{i+1}-2n)^{\d}} -\tfrac{C_1b^2}{L_{i+1}-2n} \right)}
\leq \tfrac{2C_1t^2}{L_{i+1}-2n}.
$$
Hence \eqref{unidim} follows for all $t,n,L$. This ends the proof of Proposition \ref{laptr}.
 \end{proof}
We remark that Proposition \ref{laptr} is not necessarily sharp, since in
the regime where $L_{i+1}-2n$ is of order $L_{i+1}$, one has that $\D$ behaves as $L_{i+1}^{-1}$, 
and in analogy with Gaussian behavior one should expect $t^2/L_{i+1}^2=t^2\D^2$ rather than $t^2\D$ as in \eqref{laptr1}. However,
this bound is sufficient for our purposes and is weak enough to hold throughout the whole
range of values of $t,n,L$ that we are considering.

        \subsection{ Logarithmic Sobolev inequality for adjacent transpositions
        }\label{interch}
        In this subsection we prove Theorem \ref{AT_th}.  It will be convenient to prove a more general estimate related to the so called interchange process. The latter is defined as follows. 
        Given a graph $G=(V,E)$, with $|V|=n$ vertices, the interchange process on $G$ is the continuous-time Markov chain 
with state space  $S_n$ the symmetric group of permutations of $n$ objects, defined by the infinitesimal generator
\begin{equation}
\label{ipgen}
\cL_G f(\si)=\sum_{\{x,y\}\in E} \grad_{x,y}f(\si)\,,
\end{equation}
where $\si\in S_n$, $f: S_n\mapsto \bbR$, and 
$\grad_{x,y}f(\si)=f(\si^{x,y})-f(\si)$, if $
\si^{x,y}$ denotes the permutation $
\si'$ obtained from $\si$ after the $\{x,y\}$-transposition.
If we think of $\si$ as an assignment of $n$ labels to the vertices $V$, then we interpret the process as follows: independently, with rate one, each edge $e\in E$ swaps the labels at its end points.
The Dirichlet form is given by
\begin{equation}
\label{ipdir}
\cE_G(f,g)=-\u[f\cL_G g]=
\frac12\sum_{\{x,y\}\in E} \u[\grad_{x,y}f\grad_{x,y}g
]\,,
\end{equation}
where $\u$ denotes 
the uniform probability distribution on $S_n$ and $f,g: S_n\mapsto\bbR$.
The logarithmic Sobolev constant for the interchange process on $G$ is defined as 
\begin{equation}
\label{logsobip}
\a(G)=\inf_f\frac{\cE_G(\sqrt f,\sqrt f)}{\ent_\u(f)}
\end{equation}
where $f$
ranges over all $f:S_n\mapsto \bbR_+$, and 
$\ent_\u(f)=\u\left[f\log (f/\pi[f])\right].$
We will need the following result. 
\begin{theorem}\label{gln}
Let $\G_n$ be the $n$-segment, i.e.\ the  graph with $V=\{1,2,\dots,n\}$ and $E=\{\{i,i+1\}, i=1,\dots,n-1\}$. There exists $c>0$ such that for any $n\in\bbN$ :
 \begin{equation}
\label{gln1}
\a(\G_n)\geq \frac{c}{n^{2}}
\end{equation}
\end{theorem}

Before proving the theorem, let us show that it is indeed sufficient to establish Theorem \ref{AT_th}.
Given $2d$ non negative integer numbers $n_1,\dots,n_{2d}$ such that $\sum_{i=1}^{2d}n_i=n$, and a permutation $\pi\in S_n$, $\pi(1),\dots,\pi(n)$, let us define $\o(i)=j$ iff $\pi(i)\in[n_1+\dots n_{j-1}, n_1+\dots n_{j}]$. In words, if we think of $\pi(i)$ as a label at vertex $i\in V$, then we paint labels with $2d$ colors in such a way that the first $n_1$ labels have color $1$, the next $n_2$ labels have color $2$ and so on.  Thus for a fixed choice of $n_i$'s as above we have a map $S_n\mapsto\O(n_1,\dots,n_{2d})$, where $\O(n_1,\dots,n_{2d})$ is the set of $\o\in\{1,\dots,2d\}^n$ such that 
$\sum_{i=1}^n\ind_{\o(i)=j} = n_j$, for all $j=1,\dots,2d$. 

Projecting the interchange process along the map described above gives the continuous time Markov chain on $\O(n_1,\dots,n_{2d})$ with infinitesimal generator
\begin{equation}
\label{blgen}
\cG_G h(\o)=\sum_{\{x,y\}\in E} \grad_{x,y}h(\o)\,,
\end{equation}
where $h: \O(n_1,\dots,n_{2d})\mapsto \bbR$, and $ \grad_{x,y}h(\o)=h(\o^{x,y})-h(\si)$, where $\o^{x,y}$ denotes the element $\o'\in\O(n_1,\dots,n_{2d})$ such that $\o(i)=\o'(i)$ for all $i\neq x,y$ and $(\o'(x),\o'(y))=(\o(y),\o(x))$. The associated Dirichlet form is given by
\begin{equation}
\label{excldir}
\cD_G(\varphi,\psi)=-\hat\u[\varphi\,\cG_{G} \psi]=
\frac12\sum_{\{x,y\}\in E} \hat\u[\grad_{x,y}\varphi\grad_{x,y}\psi]\,,
\end{equation}
where $\varphi,\psi:\O(n_1,\dots,n_{2d})\mapsto\bbR$, and $\hat\u$ denotes the uniform distribution on $\O(n_1,\dots,n_{2d})$.
In the special case $d=1$, this is the exclusion process on the graph $G$ with $n_1$ particles and $n_2=n-n_1$ empty sites. 
In general, we define the  logarithmic Sobolev constant for the above process as 
\begin{equation}
\label{logsobexcl}
\a(G,n_1,\dots,n_{2d})=\inf_h\frac{\cD_{G}(\sqrt h,\sqrt h)}{\ent_{\hat\u}(h)},
\end{equation}
where 
the infimum ranges over $h:\O(n_1,\dots,n_{2d})\mapsto \bbR_+$
Note that by contraction, one has \begin{equation}
\label{contrax}
\a(G)\leq \a(G,n_1,\dots,n_{2d}).
\end{equation} 
We remark that for any graph $G$, the above described projection is known to leave the spectral gap of the process invariant \cite{CaLiR}. In contrast, the log-Sobolev constant can be changed by the projection, that is the inequality \eqref{contrax} can be strict. For instance, if $G$ is the complete graph $K_n$, then
Lee and Yau \cite{LeY} (see also \cite{DiSa} for some earlier results)
prove that 
 \begin{equation}
\label{LY1}
\g(K_n)\asymp\frac{n}{\log n}\,
\end{equation}
where the symbol $a_n\asymp b_n$ means that $c
\leq (b_n/a_n)\leq c^{-1}$, 
for an absolute constant $c>0$. Moreover, they also prove that for $d=1$, any $n_1\in\{1,\dots,n-1\}$, 
 \begin{equation}
\label{LY2}
\a(K_n,n_1,n-n_1)\asymp\frac{n}{-\log (\tfrac{n_1}n(1-\tfrac{n_1}n))}\,.
\end{equation}
In particular, if $n_1\asymp n/2$ one has $\a(K_n,n_1,n-n_1)\asymp n$.  

Consider now the special case $n=L\in2\bbN$, and $n_i=n_{i+d}$, for all $i=1,\dots,d$, $\sum_{i=1}^dn_i=L/2$. Then $\O(n_1,\dots,n_{2d})$ coincides with our set $\O_L$ of $\bbZ^d$ paths of length $L$
started at the origin which come back to the origin after $L$ steps.  Moreover, the measure $\mu^{(d)}$ coincides with $\hat \pi$ and 
$\ent_d(f) = \ent_{\hat\pi}(f)$ for all $f\geq 0$. Therefore, using \eqref{contrax} and Theorem \ref{gln} one finds
 \begin{equation}
\label{entd}
\ent_d(f)\leq CL^2\sum_{x=1}^{L-1}\mu^{(d)}\left[(\nabla_{x,x+1}\sqrt f)^2\right].
\end{equation}
Taking expectation with respect to the measure $\mu$ one obtains the estimate \eqref{recur20}.
Thus, we have checked that Theorem \ref{gln} implies Theorem \ref{AT_th}.

\subsection{Proof of Theorem \ref{gln}}
For the graph $\G_n$ it is well known, see \cite{Yau_gen}, that uniformly in $n_1\in\{1,\dots,n-1\}$:
 \begin{equation}
\label{N2}
\a(\G_n,n_1,n-n_1)\asymp n^{-2}\,.
\end{equation}
That is, the simple exclusion process on the $n$-segment  has log-Sobolev constant scaling like $n^{-2}$ independently of the number of particles $n_1$. In particular, by \eqref{contrax}, $\a(\G_n)\leq Cn^{-2}$.
We want to establish the lower bound $\a(\G_n)\geq C^{-1}n^{-2}$. We could not find an explicit proof of this statement in the literature.  We shall derive this estimate as a consequence of \eqref{LY1} and \eqref{LY2} together with a suitable recursive argument that might be of interest on its own. 
We remark that a simple comparison argument between Dirichlet forms, see e.g.\ \cite{DiSa}, gives the estimates 
\begin{equation}\label{compa0}
\cE_{K_n}(f,f)\leq C\,n^3\cE_{\G_n}(f,f)\,,
\end{equation}
for some constant $C>0$ and all functions $f$. Thus, \eqref{LY1} implies $\a(\G_n)\geq c\,(n^{2}\log n)^{-1}$, which is not sufficient for our purpose.

We use the short hand notation $\a(n)$ for $\a(\G_n)$, $\cE_n$ for $\cE_{\G_n}$. 
Fix $n_1\in \{1,\dots,n-1\}$, $\si\in S_n$ and let $X=X(n_1,\si)$ denote the
vector in $\{0,1\}^n$ such that $X_i=1$ iff label $i$ occupies one of the first $n_1$ vertices of $V=\{1,\dots,n\}$.  
Clearly, as $\si$ spans $S_n$, $X$ spans $\O(n_1,n-n_1)$. 
Let $\u_X=\u[\cdot\tc X]$ denote the distribution $\u$ conditioned on the value of $X$. The entropy of a function $f:S_n\mapsto\bbR_+$ can be decomposed as
 \begin{equation}
\label{gln2}
\ent_\u(f)=\u[\ent_{\u_X}(f)] + \ent_\u(\u[f\tc X]).
\end{equation}
Note that $\u_X$ is a
product measure over the product space $S_{n_1}\times S_{n-n_1}$. Thus,
by the well known tensorization property of entropy, see e.g.\ \cite{DiSa}, one has
 \begin{equation}
\label{gln3}
\ent_{\u_X}(f)\leq (\a(n_1)\wedge\a(n-n_1))^{-1} \Big\{\tfrac12
\sum_{i=1}^{n_1-1}\u_X[(\grad_{i,i+1}	\sqrt f)^2]
+\tfrac12 \sum_{i=n_1 +1}^{n -1}\u_X[(\grad_{i,i+1}\sqrt f)^2]
\Big\}.
\end{equation}
Taking expectation w.r.t.\ $\u$ in \eqref{gln3}, and using $\u[\u_X(g)]=\u[g]$ for any $g$, one has that 
 \begin{equation}
\label{gln4}
\u[\ent_{\u_X}(f)]
\leq (\a(n_1)\wedge\a(n-n_1))^{-1}\cE_{n}(\sqrt f,		\sqrt f)
\end{equation}
To estimate the second term in \eqref{gln2}, note that the marginal of $\u$ on $X$ coincides with the 
 uniform distribution on $\O(n_1,n-n_1)$, which we will denote by $\hat \pi$.
From \eqref{LY2}, setting $\varphi(X)=\u[f\tc X]$ we can estimate
 \begin{equation}
\label{gln5}
\ent_\u(\u[f\tc X]) = \ent_{\hat\pi}(\varphi)
\leq C\left(\frac{n}{-\log (\tfrac{n_1}n(1-\tfrac{n_1}n))}
\right)^{-1} \sum_{1\leq i<j\leq n} \hat\pi[(\grad_{i,j}\sqrt\varphi)^2].
\end{equation}
Observe that if $ X^{i,j}$ denotes the configuration $X$ after the swap of $\{i,j\}$, one has
$$\u[f\tc X^{i,j}]=\u[f^{i,j}\tc X].$$ 
Thus,
$$
\grad_{i,j}\sqrt\varphi
= \sqrt {\u[f^{i,j}\tc X]}- \sqrt{ \u[f\tc X]}.
$$
 Convexity of $(0,
 \infty)^2\ni(a,b)\mapsto (\sqrt a - \sqrt b)^2$ implies that
 $$
 (\grad_{i,j}\sqrt\varphi)^2\leq  \u\left[(\sqrt{f^{i,j}} - \sqrt f)^2
 \tc X\right].
 $$
 Therefore, using $\hat\pi[\pi(\cdot\tc X)] = \pi$:
 $$
 \sum_{1\leq i<j\leq n}  \hat\pi[(\grad_{i,j}\sqrt\varphi)^2]\leq  \sum_{1\leq i<j\leq n} \u\left[(\sqrt{f^{i,j}} - \sqrt f)^2
\right] = \cE_{K_n}(\sqrt f,\sqrt f).
 $$
Using \eqref{compa0} one can estimate \eqref{gln5} by
 \begin{equation}
\label{gln6}
\ent_\u(\u[f\tc X]) 
\leq C\,n^2\left(-\log (\tfrac{n_1}n(1-\tfrac{n_1}n))
\right) \cE_n(\sqrt f,\sqrt f).
\end{equation}
Then, from \eqref{gln2} and \eqref{gln4}, one has
 \begin{equation*}
\a(n)^{-1}\leq 
(\a(n_1)\wedge\a(n-n_1))^{-1} + C\,n^2\left(-\log (\tfrac{n_1}n(1-\tfrac{n_1}n))
\right).
\end{equation*}
Up to now $n_1$ was arbitrary. We may take $n_1=\lfloor n/2\rfloor$
to obtain
\begin{equation}
\label{gln7}
\a(n)^{-1}\leq 
(\a(\lfloor n/2\rfloor)\wedge\a(n-\lfloor n/2\rfloor))^{-1} + C\,n^2,
\end{equation}
for a new constant $C>0$.
Iterating  \eqref{gln7} one arrives easily at $\a(n)^{-1}=O(n^2)$. 
\hfill $\qed$

       \section{Proof of the lower bound. }\label{lower}
         Let $h(\eta) \in \bbZ$ be the projection of a vector $\eta\in\O_L$ on
          its first coordinate. Then $h(\eta)\in \O^1_L$, where 
          $$\gO^{1}_{L} := \{ \phi\in \bbZ^{L+1}:\;\phi_{0}= \phi_{L} = 0, \;\phi_{x}-\phi_{x-1}
          \in \{-1,0,1\}  \}.$$  
          Notice that $h(\eta_t)$ is not a  Markov chain under the evolution $\eta_t$ defined by \eqref{dyn}. However, the following lemma allows us to describe the evolution of a linear function of the field $h(\eta)$.       Let $\Delta$ be the discrete Laplace operator
               \begin{equation*}
                (\Delta \phi)_{x} = \frac{1}{2}(\phi_{x-1}+ \phi_{x+1}) - \phi_{x},
               \end{equation*} 
             and define  $g_{x} := \sin\left( \frac{\pi x}{L}\right)$. For any $x=\{1,\ldots,L-1 \}$, 
              \begin{equation}\label{eigen}
               (\Delta g)_{x} = - \kappa_{L} g_{x},  
              \end{equation} 
              where $\kappa_L$ is the principal Dirichlet eigenvalue of $\D$ given by $\kappa_L= 1 - \cos(\frac{\pi}{L})$. Notice that $\k_L\sim \tfrac{\pi^2}{2L^2}$. For any $\eta\in \gO_{L}$, we define the function $$\Phi(\eta) = \sum_{x=1}^{L-1} g_{x} h_{x}(\eta),$$
                where we use the notation $h_x$ for the map $\eta\mapsto h(\eta)_x = \eta_x\cdot e_1$. 
          \begin{lem}\label{wil1}
          Let $\cL$ be the generator \eqref{dyn}. 
          Then for all $\eta\in\O_L$: \begin{equation}
               \cL \Phi(\eta) =  - \kappa_{L} \Phi(\eta).  
              \end{equation}       
                  \end{lem} 
        \begin{proof}
        Observe that for any given $y\in\{1,\dots,L-1\}$ one has that the $d$-dimensional vector $\eta_y$ satisfies, coordinatewise
        $$
        \cL \eta_y = \tfrac12(\eta_{y-1}+\eta_{y+1}) - \eta_y = (\D \eta)_y.
        $$
        By projecting along the first coordinate, the same expression holds for the function $h_y:\eta\mapsto h(\eta)_y$,
        \begin{equation}\label{deltaphi}
        \cL h_y(\eta) = \tfrac12(h_{y-1}(\eta)+h_{y+1}(\eta)) - h_y(\eta) = (\D h)_y(\eta).
        \end{equation}
Using linearity, summation by parts, and \eqref{eigen} conlcudes the proof.
        \end{proof}
        Let $P_t = e^{t\cL}$ denote the semigroup of the process, so that for any $\eta_0\in\O_L$, any function $f:\O_L\mapsto\bbR$, one has that the configuration $\eta_t$ at time $t$ with initial state  $\eta_0$ satisfies 
        $\bbE f(\eta_t) = P_t f (\eta_0)$. 
         It follows from Lemma \ref{wil1} that for all $t\geq 0$
         \begin{equation}\label{eigen2}
         \bbE[\Phi(\eta_t)] = P_t\Phi(\eta_0)= e^{-\kappa_Lt}\Phi(\eta_0).
         \end{equation}
         As $t\to\infty$ one has $ \bbE[\Phi(\eta_t)] \to\mu[\Phi]=0$. Consider now the evolution with initial state at the configuration $\eta_0=\eta^*$ defined by $(\eta^*)_x=x e_1$, for $x\in\{0,\dots,L/2\}$ and $(\eta^*)_x=(L-x) e_1$, for $x\in\{L/2 +1,\dots,L\}$, that is the maximal configuration for the first coordinate. 
        Then the initial value $\Phi(\eta^*)$ is of size $L^2$ and therefore for time $t\leq c L^2\log L$, for a suitable constant $c>0$, \eqref{eigen2} tells us that $  \bbE[\Phi(\eta_t)]$ is much larger than its equilibrium value $0$. 
       We can use this fact to lower bound the total variation distance from equilibrium. However, this only allows us to prove that the mixing time is at least $c L^2$ because the $L_\infty$ norm of $\Phi$ is also of size $L^2$. Indeed, observing that $|\Phi|_\infty = \Phi(\eta^*)$, one has
        $$
        \|p_{t}(\eta^*,\cdot)- \mu\|_{\rm TV}\geq \tfrac12(|\Phi|_\infty )^{-1}\,P_{t}\Phi(\eta^*) =\tfrac12e^{-\kappa_Lt}\geq \tfrac12\,e^{- Ct/L^2},
        $$
        for some constant $C>0$, where we use 
        $\|\nu-\nu'\|_{\rm TV} = \tfrac12\sup_{f:\;|f|_\infty\leq 1}(\nu(f)-\nu'(f))$.
        
        To obtain the extra logarithmic factor in the lower bound we follow Wilson's approach in \cite{Wil} and compute the variance of the random variable $\Phi^*_t:=\Phi(\eta_t)$ when the initial state is $\eta_0=\eta^*$. 
          \begin{lem}\label{wil2}
          The random variable $\Phi^*_t$ satisfies
\begin{equation}\label{varbo}
               \var(\Phi^*_t)\leq C_0 L^3\,,  
              \end{equation}       
              for some constant $C_0>0$, for all $t\geq 0$. 
                  \end{lem}
                  
Before proving the lemma,  let us conclude the proof of the lower bound in Theorem \ref{mainmix}.  
From Lemma \ref{wil2} and Chebyshev's inequality we have
\begin{equation}\label{chebs}
            \bbP\left[| \Phi^*_t - \bbE[\Phi^*_t]| \geq \sqrt{L^3/\e}\right] \leq C_0\,\e.
           \end{equation}   
           Let $E=\{\eta\in\O_L:\;\Phi(\eta)\geq \sqrt{L^3/\e}\}$, for some $\e>0$ to be fixed below. 
           In the limit $t\to\infty$, using $\bbE[\Phi^*_t]\to \mu[\Phi] = 0$, \eqref{chebs} yields the estimate
          \begin{equation}\label{chebs1}
            \mu(E)=\bbP\left[\Phi^*_\infty \geq \sqrt{L^3/\e}\right]\leq C_0\,\e.
           \end{equation} 
           Moreover, if $T$ is such that $\bbE[\Phi^*_T]\geq 2\sqrt{L^3/\e}$, then \eqref{chebs} implies 
           \begin{equation}\label{chebs2}
            p_T(\eta^*,E)=\bbP\left[\Phi^*_T \geq \sqrt{L^3/\e}\right] \geq 1- C_0\,\e.
           \end{equation} 
           Thus, fixing $\e=1/(4C_0)$, one has 
           $$
           \|p_{T}(\eta^*,\cdot)- \mu\|_{\rm TV}\geq p_T(\eta^*,E)-\mu(E) \geq \tfrac12,
           $$
           so that $\tmix\geq T$. On the other hand, from \eqref{eigen2},
           $$
\bbE[\Phi^*_T]=e^{-\kappa_L T}\Phi(\eta^*)\geq c_1L^2 e^{-T/c_1 L^2},
           $$ 
           for some constant $c_1>0$, so that $T=c_2 L^2\log L$ for some new constant $c_2$ suffices.
           This concludes the proof of the lower bound in Theorem \ref{mainmix}.  It remains to prove Lemma  \ref{wil2}.
           
           \subsection{Proof of Lemma \ref{wil2}}
           We start by proving that for any $\eta\in\O_L$: 
            \begin{equation}\label{genphi2}
\cL\Phi^2(\eta) \leq -2\kappa_L\Phi^2(\eta) + C\,L.
           \end{equation} 
           Notice that by \eqref{deltaphi} if $x\neq y$, then 
            \begin{equation}\label{genphi3}
           \cL h_x h_y = h_x(\Delta h)_y + h_y(\Delta h)_x.
           \end{equation}
           On the other hand a simple computation shows that
            \begin{equation}\label{genphi4}
           \cL h_x^2 = \ind_{|h_{x-1}-h_{x+1}|\neq 2}(\D h^2)_x + \tfrac1d\,\ind_{\eta_{x-1}=\eta_{x+1}}
           \end{equation}
           Let $\e_x:=\ind_{|h_{x-1}-h_{x+1}|\neq 2}$ and $\d_x:=\tfrac1d\ind_{\eta_{x-1}=\eta_{x+1}}$. 
           Then, \eqref{genphi3} and  \eqref{genphi4} yield
            \begin{equation*}
                 \cL\Phi ^2 = \sum_{x=1}^{L-1} g^{2}_{x} \left(\e_x(\Delta h^{2})_{x}  
                 + \d_x\right) +2 \sum_{x \neq y} g_{x} g_{y} h_y(\Delta h)_x      .
              \end{equation*} 
Now we observe that
\begin{equation*}
    \begin{aligned}
       \sum_{x\neq y} g_{x} g_{y}h_y(\Delta h)_x & =   \sum_{x} g_{x}
      (\Delta h)_{x} \sum_{y:\;y\neq x} g_{y} h_{y} \\
       & =  \sum_{x} g_{x}
      (\Delta h)_{x} (\Phi - g_{x} h_{x}) \\
       & =  \Phi \sum_{x} g_{x}
      (\Delta h)_{x} -  \sum_{x} g^{2}_{x} h_{x} (\Delta h)_{x}. 
    \end{aligned}
    \end{equation*} 
    Summing by parts, from \eqref{eigen} 
     we infer that:
      \begin{equation}\label{aboe}
      \cL\Phi ^2 = -2 \kappa_{L} \Phi^{2} + \sum_{x=1}^{L-1} g_{x}^{2} \left(\e_x(\Delta h^{2})_{x} + \d_x-2 h_{x} (\Delta h)_{x}   \right). 
      \end{equation} 
      Notice that if $|h_{x-1}-h_{x+1}|=2$, then necessarily $(\D h)_x=0$. Therefore we may replace  $h_{x} (\Delta h)_{x}$ by $ \e_xh_{x} (\Delta h)_{x}$ in \eqref{aboe}. 
      For any $x$:
       \begin{equation*}
        \begin{aligned}
          (\Delta h^{2})_{x} -2 h_{x} (\Delta h)_{x} & = \tfrac{1}{2}(h_{x+1}^{2} 
          + h_{x-1}^{2}) -h_{x}^{2}  -2 h_{x} \left(\tfrac{1}{2}(h_{x+1} + h_{x-1}) -h_{x}\right) \\
          & = \tfrac{1}{2} \left(h_{x+1}^{2} 
          + h_{x-1}^{2} \right) - h_{x}h_{x+1} - h_{x}h_{x-1} + h_{x}^{2} \\
           & = \tfrac{1}{2} \left[ (\grad h)^2_{x} + (\grad h)_{x+1}^{2}\right].
        \end{aligned}
       \end{equation*} 
       Summarizing
       $$
       \cL\Phi ^2 = -2 \kappa_{L} \Phi^{2} + \sum_{x=1}^{L-1} g_{x}^{2} \left(\tfrac12\,\e_x\left[ (\grad h)^2_{x} + (\grad h)_{x+1}^{2}\right] + \d_x  \right). 
       $$
Using $|\grad h|\leq 1$ and $\d_x\leq 1$ one has the desired bound \eqref{genphi2}.
           
          Next, using $\frac{d}{dt}P_{t} = P_t\cL$ and \eqref{genphi2} one has 
            \begin{equation*}
            \frac{d}{dt}\left[e^{2 \kappa_{L}t} P_{t}\Phi^{2}\right] \leq C\, L \,e^{2 \kappa_{L} t}.
          \end{equation*}  
         Therefore
            \begin{equation*}
            e^{2 \kappa_{L}t}P_{t}\Phi^{2} \leq  C\,\frac{L}{ 2\kappa_{L}} \left(e^{2 \kappa_{L} t} -1 \right) + \Phi^{2}.
            \end{equation*} 
    Recalling \eqref{eigen2}, we then obtain:
           \begin{equation*}
            \var( \Phi^*_{t}) = P_t \Phi^2 (\eta^*) -  (P_t \Phi (\eta^*))^2\leq C\,\frac{L}{2 \kappa_{L}}\left(1 - e^{-2 \kappa_{L} t} \right) \leq C\frac{L}{2 \kappa_{L}}.
           \end{equation*} 
  Thus, \eqref{varbo} holds for a suitable constant $C_0$.
  \hfill $\qed$

            \appendix
            
            \section{On the distribution of the number of particles}
            In this section we show that for a system of size $L$ the number of particles of a given type  behaves roughly like a gaussian variable with mean $L/2d$ and variance proportional to $L$.  
            Recall the definition of the condition {\rm Conv}$(c,\bar n$) from Section \ref{secMi}. Fix $L\in 2\bbN$ and the dimension $d$. 
            Define 
            \begin{equation}\label{lawn1}
     \g(n):= \mu(N_{1} = n)\,. 
   \end{equation}

      \begin{prop}\label{birthanddeathfin}   
  The measure $\g$ defined in \eqref{lawn1} satisfies {\rm Conv}$(c,\bar n$) for some constant $c>0$ independent of $L$, with $\bar n:=L/2d$. 
      \end{prop}
      \begin{proof}
      We derive the proposition from a local central limit theorem for sums of independent random variables.      Define
the partition function    
\begin{equation}\label{partition}
    Z_{L}^{d} = \sum_{n_{1},\ldots,n_{d} \in \Z_{+}, \;\sum_{j=1}^{d}n_{j} = L/2} \frac{L!}{(n_{1}!)^{2} \ldots (n_{d}!)^{2}}
   \end{equation} 
where the sum is over all nonnegative integers $n_{1},\ldots,n_{d}$ such that $\sum_{j=1}^{d}n_{j} = L/2$. Notice that $Z_{L}^{d} =|\O_L|$ is the number of $\bbZ^d$ paths starting and ending at the origin with length $L$. 
Let $p_{2n}$ denote the probability that the simple random walk on $\bbZ^d$ started at the origin is at the origin after $2n$ steps. 
     Then 
       \begin{equation}\label{srw}
Z_L^d   = (2d)^L p_L      .  \end{equation} 
The local central limit theorem, see e.g.\ \cite[Theorem 2.1.3]{LL}, states that 
   \begin{equation}\label{srw1}
  | p_{2n} - 2\,\bar p_{2n}| = O(n^{-(d+2)/2})       \end{equation} 
       where $\bar p_{2n} = (4\pi n)^{-d/2} ({\rm det }\G)^{-1/2}$, where $\G$ is the covariance matrix of the simple random walk. 
       The latter is given by $\G_{ij} = \d_{ij}\tfrac1d$, so that ${\rm det }\G = d^{-d}$, and therefore 
       $$\bar p_{2n} = (4\pi n)^{-d/2}d^{d/2} .$$
       It follows from \eqref{srw1} that 
          \begin{equation}\label{srw2}
  p_{2n} = (1+o(1))2\,\bar p_{2n} = 2 d^{d/2}(2\pi)^{-d/2}(2n)^{-d/2}.       \end{equation} 
Taking $L=2n$, from \eqref{srw},
  \begin{equation}\label{srwo}
Z_L^d   = (1+O(1/L))Q_d\,(2d)^L L^{-d/2},      \end{equation} 
where $Q_d:=2 d^{d/2}(2\pi)^{-d/2}$. 
%

    From \eqref{lawincre}, the distribution of 
   $N_{1}$ is given by
      \begin{equation}\label{lawN1}
   \begin{aligned}
     \gga(k) = \mu(N_{1} = k) =  \frac{L!}{(k!)^{2} (L-2k)!} \frac{Z_{L-2k}^{d-1}}{Z_{L}^{d}},
   \end{aligned}
   \end{equation} 
    where $0 \leq k \leq L/2$.
          To prove \eqref{conc}, we check that for $n \in [-\frac{L}{2d},(1-\tfrac1d)\frac{L}{2}]$, one has          \begin{equation}\label{dif}
         \frac{c\,e^{-\frac{n^{2}}{cL}}}{\sqrt{L}}  \leq \gga(n+\tfrac{L}{2d}) \leq \frac{ e^{-\frac{cn^{2}}{L}}}{c\sqrt{L}}.
          \end{equation} 
         Set $c_d:=(1-1/d)$.
         Clearly, 
              \begin{equation*}\label{nle}
       \gga(n+\tfrac{L}{2d}) = \frac{L!}{\left((n+\tfrac{L}{2d})!\right)^{2} \left(c_dL-2n\right)!} \frac{ Z^{d-1}_{c_dL-2n}}{Z_{L}^{d}}.
      \end{equation*}    
      We use Stirling's formula, in the form 
           \begin{equation}\label{stir}
       n! = \sqrt {2\pi}\,n^{n+\tfrac12}e^{-n}\Big(1+
       O\big(\tfrac1{n}\big)\Big).
      \end{equation}    
From \eqref{srwo} and \eqref{stir}, we obtain
               \begin{align}\label{nles}
       &\gga(n+\tfrac{L}{2d}) = \big(1+R(n,L)\big)
      \frac{Q_{d-1}}{Q_d} 
   \frac{(2(d-1))^{c_dL-2n}L^{d/2} }{(2d)^{L}(c_dL-2n)^{(d-1)/2} \left((n+\tfrac{L}{2d})!\right)^{2} \left(c_dL-2n\right)!
   }\nonumber\\
    &\quad =\big(1+R(n,L)\big)\frac{2dQ_{d-1}}{c_d^{d/2}Q_d} \frac1{\sqrt L}
     \left(1+\tfrac{2dN}{L}\right)^{-\tfrac{L}{d}-2k-1}
      \left(1-\tfrac{2N}{c_dL}\right)^{-c_dL+2N-d/2},
\end{align}    
where $$R(n,L) = O\big(\max\big\{(n+\tfrac{L}{2d})^{-1},(c_dL-2n)^{-1}\big\}\big).$$ Expanding in \eqref{nles} and simplifying one has
\begin{align}\label{nles2}
       \gga(n+\tfrac{L}{2d}) &= \big(1+R(n,L)\big)
       \tfrac{C(d)}{\sqrt L}
   \exp\left( -\tfrac{2d^{2}}{d-1} \tfrac{n^{2}}{L} + O(\tfrac{n^3}{L^2})\right),
\end{align}    
for some constant $C(d)>0$.
In particular, \eqref{nles2} shows that \eqref{dif} holds for all $n\in[-L^{2/3},L^{2/3}]$.
It remains to check the claim when $|n|\geq L^{2/3}$. We observe that
$$
\gga(n+\tfrac{L}{2d}) = \frac{\bbP(N_1 = n+\tfrac{L}{2d}; S_L=0)}{\bbP(S_L=0)},
$$
       where $S_L$ is the simple random walk on $\bbZ^d$ started at the origin after $L$ steps, and $N_1$ 
       is the number of $+e_1$ increments. Now, for some constant $C_1>0$, $\bbP(S_L=0) = p_L\geq (C_1L)^{-d/2}$ from \eqref{srw1}, so that 
       $$
       \gga(n+\tfrac{L}{2d}) \leq  C_1L^{d/2}\bbP(N_1 = n+\tfrac{L}{2d}).
       $$
       Under the law $\bbP$, $N_1$ is a binomial random variable with parameters $1/2d$ and $L$, and therefore, by Hoeffding's inequality,  any $n$ such that $|n|\geq L^{2/3}$ satisfies
       $$
       \gga(n+\tfrac{L}{2d}) \leq  C_1L^{d/2}e^{-2n^2/L}\leq 
       \frac{ e^{-\frac{cn^{2}}{L}}}{c\sqrt{L}},
       $$
       for a suitable constant $c>0$. We are left to show that 
      \begin{align}\label{nles3}        \gga(n+\tfrac{L}{2d}) 
      \geq \frac{c\,e^{-\frac{n^{2}}{cL}}}{\sqrt{L}} ,
       \end{align}
       for some $c>0$ for all $|n|\geq L^{2/3}$. We prove it for $n\geq L^{2/3}$, the case $n\leq -L^{2/3}$ being similar. 
Notice that, from \eqref{lawN1} with $k=\ell+L/2d$, 
\begin{align}\label{nles5}
\frac{\gga(k+1)}{\gga(k)} &= \frac{(c_dL-2\ell)(c_dL-2\ell-1)}{(\ell+L/2d+1)^{2} } \frac{Z_{c_dL-2(\ell+1)}^{d-1}}{Z_{c_dL-2\ell}^{d-1}}.
\end{align}
Using \eqref{srwo} and simplifying
\begin{align*} 
\frac{\gga(k+1)}{\gga(k)} & \geq\big(1-\tfrac{C}{c_dL - 2\ell}\big)\frac{\big(1-\tfrac{2\ell}{c_dL}\big)\big(1-\tfrac{2\ell+1}{c_dL}\big)}{\big(1+\tfrac{2d(\ell+1)}{L}\big)},
\nonumber\end{align*}
for some constant $C>0$. 
Taking products over $\ell=0,\dots,n-1$, 
\begin{align*}\label{nles6}
\gga(n+\tfrac{L}{2d}) \geq \exp{(-Cn^2/L)},
\end{align*}
for some new constant $C>0$, for all $n_0:=\e L\geq n\geq L^{2/3}$ if $\e>0$ is a suitable small constant. On the other hand, simple computations using \eqref{nles5} show that
$ \gga(n+\tfrac{L}{2d}) \geq \gga(n_0+\tfrac{L}{2d}) \exp{(-CL)}$ for all $n\geq n_0$.
This ends the proof of \eqref{dif}. 

To prove \eqref{conv1} it suffices to show that $\g(k+1)\leq C\g(k)$ for all $k\geq L/2d$ for some constant $C$. However, this follows immediately from \eqref{lawN1} and \eqref{srwo}.
The same argument proves \eqref{conv2}.
   \end{proof}
   
        \begin{lem}\label{le1}
There exists a constant $C>0$ such that for all $n\in[0,L/2-1]$:
\begin{equation}\label{mu}
\frac{n^2}{C \,(L-2n)^2}\leq \frac{\g(n)}{\g(n+1)} \leq \frac{C\,n^2}{(L-2n)^2}.
\end{equation}
\begin{proof}
This follows immediately from \eqref{lawN1}
and \eqref{srw}.
\end{proof}
\end{lem}
              \begin{lem}\label{le2}
Set $\bar N_1=N_1-L/2d$, $\si^2=\var_\mu(N_1)$ and $X=(\bar N_1)^2 - \si^2$. 
There exists a constant $C>0$ such that 
\begin{gather}\label{sig}
\frac{1}{C \,L}\leq \si^2\leq C\,L,\\
\label{vary}
\var_\mu(X)\leq C\,L^2.
\end{gather}
\end{lem}\begin{proof}
From  \eqref{dif} it follows that 
$$
\si^2 = \sum_{n} n^2\g(n+\tfrac{L}{2d}) \leq C\,L\sum_{n}\frac{n^2}{ L}\frac{e^{-\tfrac{n^2}{CL}}}{\sqrt L}\leq C'L,
$$
for some constants $C,C'>0$, where we use a comparison with integrals for $L$ large. The lower bound on $\si^2$ is obtained in the same way. 
To prove \eqref{sig} simply observe that 
$$
\var_\mu(X)\leq \mu[(\bar N_1)^4] = \sum_{n} n^4\g(n+\tfrac{L}{2d}) \leq C\,L^2\sum_{n}\frac{n^4}{ L^2}\frac{e^{-\tfrac{n^2}{CL}}}{\sqrt L}\leq C'L^2,
$$
for some constants $C,C'>0$. 
\end{proof}

    \bibliographystyle{plain}

  \bibliography{dynpol}

\end{document}